\documentclass{amsart}
\usepackage{amssymb}
\usepackage{mathrsfs}
\usepackage{graphicx}
\usepackage{comment}
\usepackage{enumitem}

\newcommand{\ds}{\displaystyle}
\newcommand{\be}{\begin{equation}}
\newcommand{\ee}{\end{equation}}
\newcommand{\ba}{\begin{align}}
\newcommand{\ea}{\end{align}}
\newcommand{\nni}{\noindent}

\newtheorem{theorem}{Theorem}[section]

\newtheorem{lemma}{Lemma}[section]

{\begin{list}{}{%

\settowidth{\labelwidth}{\textsf{{\it #1.}}}
\setlength{\labelsep}{4mm}%
\setlength{\leftmargin}{\labelwidth}
\addtolength{\leftmargin}{\labelsep}
}}
{\end{list}}

\def\beq{\begin{equation}}\def\enq{\end{equation}}

\keywords{inhomogeneous Diophantine approximation}
\date{\today}

\title{An upper bound on  the Inhomogeneous Approximation Constants}

\author[B. Paudel]{Bishnu Paudel}
\address{ Department of Mathematics\\
         Kansas State University\\
         Manhattan, KS 66506, USA}
\email{bpaudel@ksu.edu, pinner@math.ksu.edu}

\author[C. Pinner]{Chris Pinner}

\begin{document}
\begin{abstract}
For an irrational real $\alpha$  and $\gamma\not \in \mathbb Z + \mathbb Z\alpha$  it is well known that
$$ \liminf_{|n|\rightarrow \infty}  |n| ||n\alpha -\gamma || \leq \frac{1}{4}. $$
If the partial quotients, $a_i,$ in the negative `round-up' continued fraction expansion of $\alpha$
have $R:=\liminf_{i\rightarrow \infty}a_i$ odd, then the 1/4  can be replaced by
$$ \frac{1}{4}\left(1-\frac{1}{R}\right)\left(1-\frac{1}{R^2}\right), $$
which is optimal. The optimal bound for even $R\geq 4$ was already known.

\end{abstract}

\maketitle

\section{Introduction}
For a real irrational $\alpha$ and real $\gamma$ in $[0,1)$ we define the inhomogeneous approximation constant
$$ M(\alpha,\gamma):=\liminf_{|n|\rightarrow \infty}  |n| ||n\alpha -\gamma ||,  $$
where $||x||$ is the distance from $x$ to the nearest integer, and the case of worst inhomogeneous approximation
$$ \rho(\alpha) := \sup_{\gamma\not \in \mathbb Z + \mathbb Z\alpha} M(\alpha,\gamma). $$
From a well known Theorem of Minkowski, see for example  \cite[Chap. III]{Cassels} or \cite[IV.9]{Rockett},  we have
\be \label{quarter} M(\alpha,\gamma)\leq \frac{1}{4}, \ee
Grace \cite{grace} giving examples with $\rho(\alpha)=\frac{1}{4}$. These examples though, have continued fraction expansions with the 
partial quotients all going to infinity. We are interested in improving this bound when the partial quotients have a bounded
subsequence.
Instead of the usual continual fraction expansion, here we shall use the {\it negative expansion}:
\begin{equation}\label{ContExpan}
    \alpha=\frac{1}{a_1-\cfrac{1}{a_2-\cfrac{1}{a_3-\cdots}}}=:[0; a_1, a_2, a_3,\cdots]^{-},
\end{equation}
where the integers $a_i\geq2$ are generated by always rounding up instead of down
\begin{equation*}
    \alpha_0:=\{\alpha\}=\alpha, \,\ \,\ a_{n+1}:=\left\lceil{\frac{1}{\alpha_n}}\right\rceil, \,\ \,\ \alpha_{n+1}:=\left\lceil{\frac{1}{\alpha_n}}\right\rceil-\frac{1}{\alpha_n}.
\end{equation*}
An algorithm for evaluating $M(\alpha,\gamma)$ from \eqref{ContExpan} and an $\alpha$-expansion of $\gamma$
\be \label{alphaexpansion}
    \gamma=\sum_{i=1}^{\infty}\frac{1}{2}(a_i-2+t_i)D_{i-1},   \;\;\quad D_{i-1}:=\alpha_0\alpha_1\cdots \alpha_{i-1}.
\ee
was given in \cite{Pinner}; we review this in Section \ref{alg}. See \cite{Takao} for alternative approaches.
Setting
\be \label{defR}  R:= \liminf_{i\rightarrow \infty} a_i, \ee
it was also shown  in \cite{Pinner} that if $R\geq 3,$ then 
\be \label{oldbound} \rho (\alpha ) \leq \frac{1}{4}\left(1-\frac{1}{R}\right). \ee
The restriction $R\geq 3$ is needed here; the examples of Grace will have long strings of 2's in their negative expansion.
For $R\geq 4$ even, \eqref{oldbound} is best possible; for example,  if $\alpha=[0;\overline{R,2N}]^{-},$ then taking all the $t_i=0$ in \eqref{alphaexpansion} gives $\rho(\alpha)$ with
$$ \lim_{N\rightarrow \infty}\rho(\alpha) = \lim_{N\rightarrow \infty} M\left(\alpha,\frac{1}{2}(1-\alpha)\right)=\frac{1}{4}\left(1-\frac{1}{R}\right). $$
When $R$ is odd though, the bound \eqref{oldbound} can be improved.

\begin{theorem}\label{Odds}
If $R$  is odd then 
\be \label{oddbound}  \rho(\alpha) \leq C(R):=\frac{1}{4}\left(1-\frac{1}{R}\right)\left(1-\frac{1}{R^2}\right)=\frac{1}{4}\left(1-\frac{1}{R}-\frac{1}{R^2}+\frac{1}{R^3}\right). \ee
\end{theorem}

This odd $R$ bound \eqref{oddbound}   is  also best possible. 

\begin{theorem} \label{Sharps}  If $\alpha =[0;\overline{R,NR}]^-$ and the $t_i$ in the expansion of $\gamma_*$ have corresponding period $-1,N$, then 
$$\lim_{N\rightarrow \infty} M(\alpha,\gamma_*)=C(R).$$

\end{theorem} 

One can also obtain lower bounds on $\rho(\alpha)$ in terms of $R$. When $R\geq 4$ is even, we showed in \cite{BishnuInhomog} the optimal
bound
$$ \rho(\alpha) \geq \frac{R-2}{4(\sqrt{R^2-4}+1)}=\frac{1}{4}\left(1-\frac{3}{R} +\frac{5}{R^2}+O(R^{-3})\right), $$
and when $R\geq 3$ is odd, a  bound which is at least asymptotically optimal
$$ \rho(\alpha) \geq \frac{2R-2-\sqrt{(R+1)^2-4}}{4(\sqrt{(R+1)^2-4}-1)}=\frac{1}{4}\left(1-\frac{3}{R} +\frac{4}{R^2}+O(R^{-3})\right). $$

Of course in the classical homogeneous case, $M(\alpha,0)$ depends on the largest rather than the smallest partial quotients. Using the negative expansion, if $R\geq 3$,
$$\frac{1}{\sqrt{r^2-4}}\leq M(\alpha,0)\leq  \frac{1}{r-R+ \sqrt{R^2-4}},\;\;\quad r:=\limsup_{i\rightarrow \infty} a_i, $$
with equality for $\alpha=[0;\overline{r}]^-$ and $[0;\overline{r,(R,)^l}]^-$, $l\rightarrow \infty$,  respectively.
Notice $M(\alpha,0)$ only exceeds the inhomogeneous bound,  \eqref{oldbound} if $R$ is even and \eqref{oddbound} if $R$ is odd,  for a few small $r$;
namely $3\leq r\leq 7$ when $R=3$,  $r=4$ or 5 when $R=4$, and $R=5=r$.

\section{The sequence of best inhomogeneous approximations}\label{alg}

In \cite{Pinner} it was shown how to use the continued fraction expansion \eqref{ContExpan} of $\alpha$ to obtain a
unique $\alpha$-expansion of a $\gamma$ in $(0,1)$ of  the form \eqref{alphaexpansion}, where the $t_i$ are integers with 
$$ -(a_i-2) \leq t_i \leq a_i, \quad t_i\equiv a_i \mod 2, $$
with no blocks of the form $t_i=a_i$ with $t_j=a_j-2$ all $j>i$ or $t_{i+\ell}=a_{i+\ell}$ with $t_j=a_j-2$ for any $i<j<i+\ell$.

We set
\be \label{defalphas} \alpha_i=[0;a_{i+1},a_{i+2},\ldots ]^-,\quad \quad \bar\alpha_i:=[0;a_i,a_{i-1},\ldots ,a_1]^-. \ee
If $t_k=a_k$ infinitely often, then \cite[Lemma 1]{Pinner},
\be \label{ends}  M(\alpha,\gamma)\leq \liminf_{\stackrel{k\rightarrow \infty}{t_k=a_k}} \frac{\bar\alpha_k}{4(1-\bar\alpha_k\alpha_k)}. \ee

If the $t_k=a_k$ at most finitely often, then by \cite[Theorem 1]{Pinner}  it is enough to look at the sequence of $n$ of the form
$$ Q_k:=\sum_{i=1}^k\frac{1}{2}(a_i-2+t_i)q_{i-1},\quad Q_k+q_{k-1},\quad -(q_k-q_{k-1}-Q_k),  \quad -(q_k-Q_k),$$
where the $q_i$ are the convergent denominators $p_i/q_i=[0;a_1,\ldots ,a_i]^-$, $q_i=(\bar\alpha_1\cdots \bar\alpha_i)^{-1}$.

Writing
  \begin{align}\label{defd}
    d_k^-&:=t_k\bar\alpha_k+t_{k-1}\bar\alpha_k\bar\alpha_{k-1}+t_{k-2}\bar\alpha_k\bar\alpha_{k-1}\bar\alpha_{k-2}+\cdots , \\ \nonumber
    d_{k}^{+}&:=t_{k+1}\alpha_{k}+t_{k+2}\alpha_{k}\alpha_{k+1}+t_{k+3}\alpha_{k}\alpha_{k+1}\alpha_{k+2}+\cdots .
\end{align}  
we will use this result expressed in a more symmetric form.
\begin{lemma}\label{mainlemma}
If $\gamma\notin\mathbb{Z}+\alpha\mathbb{Z}$ and the $\alpha$--$expansion$ of $\gamma$ has $t_i=a_i$ at most finitely many times, then 
\begin{equation*}
    M(\alpha,\gamma)=\liminf_{k\rightarrow\infty}\min\{s_1(k), s_2(k), s_3(k), s_4(k)\},
\end{equation*}
where 
\begin{align*}
    s_1(k)&:=\frac{1}{4}(1-\bar\alpha_k+d_k^-)(1-\alpha_k+d_k^+)/(1-\bar\alpha_k\alpha_k),\\
     s_2(k)&:=\frac{1}{4}(1+\bar\alpha_k+d_k^-)(1+\alpha_k-d_k^+)/(1-\bar\alpha_k\alpha_k),\\
      s_3(k)&:=\frac{1}{4}(1-\bar\alpha_k-d_k^-)(1-\alpha_k-d_k^+)/(1-\bar\alpha_k\alpha_k),\\
       s_4(k)&:=\frac{1}{4}(1+\bar\alpha_k-d_k^-)(1+\alpha_k+d_k^+)/(1-\bar\alpha_k\alpha_k).
\end{align*}
\end{lemma}
Notice that 
$$ s_1(k)s_3(k) =\frac{\left( (1-\bar\alpha_k)^2-(d_k^-)^2\right)\left((1-\alpha_k)^2-(d_k^+)^2\right)}{16(1-\bar\alpha_k\alpha_k)^2}\leq
\frac{(1-\bar\alpha_k)^2(1-\alpha_k)^2}{16(1-\bar\alpha_k\alpha_k)^2}. $$
If $t_k=a_k$ at most finitely often, this plainly  gives
\be \label{oldrho} \rho(\alpha)\leq \liminf_{k\rightarrow \infty} \frac{(1-\bar\alpha_k)(1-\alpha_k)}{4(1-\bar\alpha_k\alpha_k)},\ee
and $\bar\alpha_k>1/R$ when $a_k=R$ readily gives \eqref{oldbound}. 
When the $a_i$ are all even we can take the $t_i=0$ and have equality in \eqref{oldrho}, but when $R$ is odd $|d_k^-|$  will not be small for the $a_k=R$ and \eqref{oldbound} can be improved.

\section{Proofs of Theorems \ref{Odds} and \ref{Sharps}}

\begin{proof}[Proof of Theorem \ref{Odds}]  Setting $\beta=[0;\overline{R}]^-$,  we can assume that $\alpha_k,\bar\alpha_k\leq \beta$ as $k\rightarrow \infty$. When $R=3$ we have $\beta=\frac{1}{2}(3-\sqrt{5})$.

We can assume $t_k=a_k$ does not occur infinitely often, since
$$ \frac{\bar\alpha_k}{4(1-\bar\alpha_k\alpha_k)}=\frac{1}{4(a_k-\bar\alpha_{k-1}-\alpha_k)} \leq \frac{1}{4(R-2\beta)} \leq \frac{1}{4\sqrt{5}}<C(3)\leq C(R). $$
This ensures that the $|d_k^-|\leq 1-\bar\alpha_k$.
Notice that
$$ s_3(k)s_4(k)= \frac{  \left( (1-d_k^-)^2-\bar\alpha_k^2\right)\left( 1 -(\alpha_k+d_k^+)^2\right) }{16(1-\bar\alpha_k\alpha_k)^2} < \frac{   (1-d_k^-)^2}{16(1-\bar\alpha_k\alpha_k)^2}. $$
Hence if $a_k=R$ and $t_k\geq 3$ then $d_k^-\geq  (3+d_{k-1}^-)\bar\alpha_k>2\bar\alpha_k$,  $\bar\alpha_k>1/R$, and 
$$ \min\{ s_3(k),s_4(k) \} \leq \frac{1-d_{k}^{-}}{4(1-\bar\alpha_k\alpha_k)}\leq \frac{1-2\bar\alpha_k}{4(1-\bar\alpha_k\alpha_k)}\leq \frac{1-\frac{2}{R}}{4(1- \frac{\beta}{R})}< C(R),$$
since $\frac{1}{R}(1-\frac{1}{R}) +\beta (1-\frac{1}{R})(1-\frac{1}{R^2})<1$.
Likewise if $t_{k}\leq -3$ using $s_1(k)$ and $s_2(k).$

So we can assume that the $a_k=R$ have $t_k=\pm 1$. Suppose now that $a_k=R$ with $t_k=-1$ (if $t_k=+1$ then we switch
$s_1(k)$ and $s_3(k)$ etc.).

\vskip0.2in
\nni
{\bf Case 1: $d_{k-1}^-$, $d_k^+\leq \frac{1}{R}.$}

Then
$$ s_1(k) \leq \frac{ \left(1-\bar\alpha_k (2-\frac{1}{R})\right)\left( 1-\alpha_k +\frac{1}{R}\right)}{4( 1-\bar\alpha_k\alpha_k)}  < \frac{1}{4}   \left(1-\bar\alpha_k \left(2-\frac{1}{R}\right)\right)\left(1+\frac{1}{R}\right) < C(R), $$
the second inequality using $\bar\alpha_k(1+\frac{1}{R})<1$ and the third using $\bar\alpha_k>\frac{1}{R}.$

\vskip0.2in
\nni
{\bf Case 2: $d_{k-1}^-$, $d_k^+\geq \frac{1}{R}.$}

Then
$$ s_3(k) \leq \frac{ \left(1-\frac{\bar\alpha_k }{R}\right)\left( 1-\alpha_k -\frac{1}{R}\right)}{4( 1-\bar\alpha_k\alpha_k)}  < \frac{1}{4}   \left(1-\frac{\bar\alpha_k }{R}\right)\left(1-\frac{1}{R}\right) < C(R), $$
the second inequality using $\bar\alpha_k(1-\frac{1}{R})<1$ and the third using $\bar\alpha_k>\frac{1}{R}.$

\vskip0.2in
\nni
{\bf Case 3: $d_{k-1}^-\leq\frac{1}{R}$, $d_k^+\geq \frac{1}{R}.$}

We observe that $d_k^-\leq -(1-\frac{1}{R})\bar\alpha_k$ and
$$ \min \{s_1(k),s_3(k) \} \leq \sqrt{ s_1(k)s_3(k)}= \frac{1}{4}\sqrt{S},\;\; $$
with
\begin{align*}S & =\frac{\left((1-\bar\alpha_k)^2-(d_k^-)^2\right)\left((1-\alpha_k)^2-(d_k^+)^2\right)}{(1-\bar\alpha_k\alpha_k)^2}\\
& \leq \frac{\left((1-\bar\alpha_k)^2-(1-\frac{1}{R})^2\bar\alpha_k^2\right)\left((1-\alpha_k)^2-\frac{1}{R^2}\right)}{(1-\bar\alpha_k\alpha_k)^2}.
\end{align*}
Hence we have
\begin{align*}
 S  & < \left((1-\bar\alpha_k)^2-\left(1-\frac{1}{R}\right)^2\bar\alpha_k^2\right)\left(1-\frac{1}{R^2}\right)\\
& < \left(\left(1-\frac{1}{R}\right)^2-\left(1-\frac{1}{R}\right)^2\frac{1}{R^2}\right)\left(1-\frac{1}{R^2}\right)=\left(1-\frac{1}{R}\right)^2\left(1-\frac{1}{R^2}\right)^2,
\end{align*}
the first inequality since $(2-\bar\alpha_k\alpha_k)\bar\alpha_k(1-\frac{1}{R^2}) +\alpha_k <2$, and the second  since $f(x)=1-2x+(\frac{2}{R}-\frac{1}{R^2})x^2$ is decreasing for $1\geq x\geq \frac{1}{R}$.

\vskip0.2in
\nni
{\bf Case 4: $d_{k-1}^-\geq \frac{1}{R}$, $d_k^+\leq \frac{1}{R}.$}

This is almost the same as Case 3, except we use
$$ \min \{s_1(k-1),s_3(k-1) \} \leq \sqrt{s_1(k-1)s_3(k-1)}= \frac{1}{4}\sqrt{S},\;\; $$
with
\begin{align*} S & =\frac{\left((1-\bar\alpha_{k-1})^2-(d_{k-1}^-)^2\right)\left((1-\alpha_{k-1})^2-(d_{k-1}^+)^2\right)}{(1-\bar\alpha_{k-1}\alpha_{k-1})^2}\\
& \leq \frac{\left((1-\bar\alpha_{k-1})^2-\frac{1}{R^2}\right)\left((1-\alpha_{k-1})^2-(1-\frac{1}{R})^2\alpha_{k-1}^2\right)}{(1-\bar\alpha_{k-1}\alpha_{k-1})^2}.
\end{align*}
The proof follows, using $\bar\alpha_{k-1}$ and $\alpha_{k-1}$ in place of $\alpha_k$ and $\bar\alpha_k$.
\end{proof}

\begin{proof}[Proof of Theorem \ref{Sharps}] If $t_k=-1$ then as $N\rightarrow \infty$
$$\bar\alpha_k\rightarrow \frac{1}{R}, \quad \alpha_k\rightarrow 0,\quad d_{k}^-\rightarrow -\frac{1}{R}+\frac{1}{R^2},\quad d_k^+\rightarrow \frac{1}{R}.$$
So
\begin{align*} s_1(k) & \rightarrow \frac{1}{4}\left(1-\frac{2}{R}+\frac{1}{R^2}\right)\left(1+\frac{1}{R}\right)=C(R),\\
s_3(k)& \rightarrow \frac{1}{4} \left(1-\frac{1}{R^2}\right)\left(1-\frac{1}{R}\right)=C(R),
\end{align*}
while $s_4(k)>s_2(k)\rightarrow \frac{1}{4}(1+\frac{1}{R^2})(1-\frac{1}{R})>C(R).$  We get the same values for
$s_1(k-1),s_3(k-1),s_2(k-1),s_4(k-1)$ and $M(\alpha,\gamma_*)\rightarrow C(R)$ as $N\rightarrow \infty$.
\end{proof}


\vspace{0cm} 
\begin{thebibliography}{99}

\bibitem{Cassels}
 J. W. S. Cassels, \text{An Introduction to Diophantine Approximation}, Cambridge Univ. Press, London/New York, 1957.
 
 




\bibitem{grace}
J.\ H.\ Grace, \textit{Note on a Diophantine approximation}, Proc.\ London Math.\ Soc.\ \textbf{17} (1918), 316-319.

\bibitem{Takao}
T.\ Komatsu, 
\textit{On inhomogeneous continued fraction expansions and inhomogeneous Diophantine approximation},
J. Number Theory \textbf{62} (1997), no. 1, 192–212. 

\bibitem{BishnuInhomog}
B.\ Paudel and C.\ Pinner, \textit{Bounding the largest inhomogeneous approximation constant},  	arXiv:2301.08825 [math.NT],  (2023).


\bibitem{Pinner}
C. G. Pinner, \textit{More on inhomogeneous Diophantine approximation}, Journal de Th\'eorie des Nombres, Bordeaux. \textbf{13} (2001) no. 2, 539-557.







\bibitem{Rockett}
 A. Rockett and P. Sz\"{u}sz, \text{Continued Fractions}, World Scientific, Singapore, 1992.

\end{thebibliography}
\end{document}